\documentclass[12pt,final]{article}
\usepackage{amsmath,amsthm,amsfonts,amssymb,graphicx,enumerate,psfrag}
\usepackage{tikz,pgfplots}
\usepackage{mathrsfs}
\usepackage{fullpage}

\usepackage[colorlinks,citecolor=blue,urlcolor=blue]{hyperref}
\definecolor{myblue}{RGB}{51,51,178}
\definecolor{myred}{RGB}{189,26,26}
\definecolor{mygreen}{RGB}{0,128,0}
\usepackage[utf8]{inputenc}

\newtheorem{theorem}{Theorem}

\newtheorem{proposition}{Proposition}
\newtheorem{corollary}{Corollary}

\newcommand{\dN}{\mathbb {N}}
\newcommand{\ip}{\textsc{ip}}
\newcommand{\ex}{\textsc{ex}}
\newcommand{\fS}{\mathfrak {S}}
\newcommand{\dZ}{\mathbb {Z}}

\newcommand{\cL}{\mathcal {L}}

\newcommand{\cE}{\mathcal {E}}

\newcommand{\cS}{\mathcal {S}}
\newcommand{\dR}{\mathbb {R}}
\newcommand{\dd}{\mathrm{d}}
\newcommand{\cH}{\mathcal{H}}

\newcommand{\EE}{{\mathbb{E}}}
\newcommand{\PP}{{\mathbb{P}}}

\newcommand{\bP}{{\mathrm{P}}}
\newcommand{\bQ}{{\mathrm{Q}}}
\newcommand{\bK}{{\mathrm{K}}}
\newcommand{\diam}{{\mathrm{diam}}}

\newcommand{\tmix}{\mathrm{t}_{\textsc{mix}}}

\newcommand{\trel}{{\mathrm{t}_{\textsc{rel}}}}

\newcommand{\dtv}{\mathrm{d}_{\textsc{tv}}}

\newcommand{\im}{{\mathrm{Im}}}
\newcommand{\Var}{{\mathrm{Var}}}

\title{Universality of cutoff for the Exclusion Process}
\author{Jonathan Hermon, Gady Kozma and Justin Salez}
\begin{document}
\maketitle
\begin{abstract}Under a fairly general condition on the underlying jump rates, we prove that the Exclusion Process with fixed particle density exhibits cutoff at time $\tmix\sim \frac{1}{2}\trel \log N$, where $N$ is the volume and $\trel$ the relaxation time of the single-particle dynamics.  Our result covers,  in particular, the standard setting of uniform nearest-neighbor jumps on large discrete tori in any fixed dimension and, more generally, on any sequence of vertex-transitive graphs with bounded degree and polynomially diverging diameter.  Our (short and entirely human) proof combines Wilson's method, the Octopus Inequality, Fourier analysis on Hamming slices, and a sharp comparison with the   Dirichlet form of a natural accelerated variant of the dynamics.
\end{abstract}
\tableofcontents

\section{Introduction}
\subsection{Model and main result}
Fix an irreducible $N\times N$ symmetric stochastic matrix $\bP$ and an integer $0<k<N$. The $k-$particle Exclusion Process with rates $\bP$ is a continuous-time Markov chain $(\cS_t)_{t\ge 0}$ on 
\begin{eqnarray}
\label{def:X}
\Omega_k & := & \left\{A\subseteq[N]\colon  |A|=k\right\},
\end{eqnarray}
whose infinitesimal generator acts on functions $f\in\dR^{\Omega_k}$ as follows:
\begin{eqnarray}
\label{def:L}
\cL^{\textsc{ex}}_{\bP,k} f(A) & := & \sum_{i\in A}\sum_{j\in A^c}\bP(i,j)\left[f(A^{ij})-f(A)\right].
\end{eqnarray}
Here  and throughout the paper, $A^{ij}$ denotes the symmetric difference of $A$ and $\{i,j\}$. 
 This  dynamics describes the  evolution of  $k$  indistinguishable particles attempting to jump independently  according to $\bP$, except that they are not allowed to overlap: any move that would violate this exclusion rule is simply censored.  We refer the unfamiliar reader to the classical references \cite{MR268959,MR2108619} for more details.
Our assumptions on $\bP$ guarantee that the uniform distribution $\pi$ on $\Omega_k$ is reversible, and that it is a global attractor for the dynamics.
The speed of convergence to equilibrium is typically measured by the \emph{worst-case total-variation  mixing time}, defined for any precision $\varepsilon\in(0,1)$ as follows:
\begin{eqnarray*}
\tmix(\varepsilon) & := & \min\left\{t\ge 0\colon \dd(t)\le \varepsilon\right\}, \quad\textrm{where}\quad \dd(t) \ := \ \max_{\cS_0\in\Omega_k}\dtv\left(\mathrm{Law}(\cS_t),\pi\right).
\end{eqnarray*} Understanding how this fundamental time-scale depends on the underlying geometry is an important problem, which has received considerable attention \cite{MR2244427,MR2629990,MR3077529,MR4164461,MR4164852,MR4546624,HP}. 

In the present paper, we determine the precise first-order asymptotics of $\tmix(\varepsilon)$ in the limit where $N$ and $k$ diverge at a comparable speed, under a fairly general condition on the rates. The latter concerns   the $N\times N$ semi-group $(\bP_t)_{t\ge 0}$ associated with  $\bP$, i.e. 
\begin{eqnarray}
\label{def:Pt}
\bP_t & := & e^{-t}\sum_{m=0}^\infty \frac{t^m\bP^m}{m!}.
\end{eqnarray}
Note that all entries of $\bP_t$ converge to $1/N$ as $t\to\infty$, and recall that an important time-scale for this dynamics is  the \emph{relaxation time} (or inverse spectral gap), defined as
\begin{eqnarray}
\trel(\bP)  & :=& \frac{1}{1-\lambda_2},
\end{eqnarray}
where $1=\lambda_1>\lambda_2\ge\ldots \ge \lambda_N$ denote the ordered eigenvalues of $\bP$.   For any fixed $u>0$, we will require that the normalized maximum return probability 
\begin{eqnarray}
M_u(\bP) & := &  N\max_{1\le i\le N}\bP_{u\trel(\bP)}(i,i),
\end{eqnarray}
 remains bounded as the matrix $\bP$ ``grows''. More precisely, we assume that  the triple $(N,k,P)$  depends on a parameter $n\in \dN$, kept implicit to lighten notation, in such a way that
 \begin{enumerate}[(i)]
\item $N  \xrightarrow[n\to\infty]{}  \infty$;
\item $\displaystyle{\frac{\log k}{\log (N-k)} \xrightarrow[n\to\infty]{} 1}$;\medskip
\item $\displaystyle{\forall u>0, \ \limsup_{n\to\infty}M_u(\bP)  <  \infty}$.
\end{enumerate}
\begin{theorem}[Main result]\label{th:main}
In the regime $\mathrm{(i)}$-$\mathrm{(ii)}$-$\mathrm{(iii)}$, we have
\begin{eqnarray}
\label{conclusion}
\forall\varepsilon\in(0,1),\qquad \frac{\tmix(\varepsilon)}{\trel(\bP)\log N}  & \xrightarrow[n\to\infty]{} & \frac{1}{2}.
\end{eqnarray}
\end{theorem}
Several important comments are in order.
\begin{enumerate}
\item  The estimate  (\ref{conclusion})  is the same as if the  particles were evolving independently according to $\bP$, despite  the highly non-trivial correlations imposed by the exclusion rule.
\item The fact that $\tmix(\varepsilon)$ is asymptotically independent of $\varepsilon$  is the signature of a remarkably abrupt transition from out-of-equilibrium to equilibrium  known as a \emph{cutoff phenomenon}. We refer the unfamiliar reader to the seminal paper \cite{MR1374011}, the classical book \cite{MR3726904} and the recent lecture notes \cite{salez2025notes} for an introduction to this fascinating subject. 
\item Our  requirement (ii) is  satisfied throughout the standard \emph{thermodynamic regime}, where the particle density approaches a limit in $(0,1)$. In fact,  the ratio $k/N$ is allowed to evolve arbitrarily with $n$, as long as it remains at distance $N^{-o(1)}$ from  $0$ and $1$.  See \cite{HP} for complementary results in the low-density regime  $k=o(\sqrt{N})$.
\item The upper bound in (\ref{conclusion}) is in fact established uniformly in $0<k<N$, and for the stronger $\chi^2$ divergence to equilibrium (see Proposition  \ref{pr:UB} below).
\item In all examples below, our proof shows that the cutoff window is in fact of order $\trel(\bP)$ whenever $k/N$ remains bounded away from $0$ and $1$. See Section \ref{sec:window}.
\end{enumerate}
Let us now show that our requirement (iii) is satisfied by many natural  geometries. 
\subsection{Examples}
 
As a first emblematic example, consider the case where $\bP$ is the transition matrix of simple random walk on the $d-$dimensional discrete torus of side-length $n$: 
\begin{eqnarray}
\label{def:tori}
\forall i,j\in\dZ^d_n,\quad \bP(i,j) & := & \left\{
\begin{array}{ll}
\frac{1}{2d} & \textrm{if } |i-j|_1 =  1\\
0 & \textrm{else.} 
\end{array}
\right.
\end{eqnarray}
In the regime where  $d\ge 1$ is kept fixed while $n\to\infty$, we classically have
\begin{eqnarray}
\label{torus:gap}
\trel(\bP) & =&  \frac{d}{1-\cos\left(\frac{2\pi}{n}\right)} \ \underset{n\to\infty}{\sim} \ \frac{dn^2}{2\pi^2},
\end{eqnarray}
while for any fixed $u>0$, the local Central Limit Theorem gives
\begin{eqnarray}
\label{torus:CLT}
n^d\,\bP_{u n^2}(0,0)\ & \xrightarrow[n\to\infty]{} &  (f_{u/d}(0))^d,
\end{eqnarray}
where $f_t$ denotes the density of $B_t\mod 1$ with $B_t\sim {\mathcal N}(0,t)$.
By symmetry, the same estimate applies to any state, and this is more than enough to guarantee the validity of our rate condition (iii). We therefore obtain the following important result, which confirms a long-standing prediction in the community of mixing times and generalizes a famous result of Lacoin \cite{MR3474475,MR3551201,MR3689972} in the case $d=1$. 
\begin{corollary}[Cutoff on $d-$dimensional tori]\label{co:tori}For any fixed $d\ge 1$, the  $k-$particle Exclusion Process with $k\in[n^{d-o(1)},n^d-n^{d-o(1)}]$ and jump rates $\bP$ as in (\ref{def:tori}) satisfies
\begin{eqnarray}
\label{torus:tmix}
\forall \varepsilon\in(0,1),\qquad \tmix(\varepsilon) & \underset{n\to\infty}{\sim}  & \frac{d^2n^2\log n}{4\pi^2}.
\end{eqnarray}
\end{corollary}

 Much more generally, consider an arbitrary finite  vertex-transitive graph  $G=(V,E)$, and let $\bP$ be the transition matrix of the simple random walk on it:
 \begin{eqnarray}
 \label{def:SRW}
\forall i,j\in V,\quad \bP(i,j) & := & \left\{
\begin{array}{ll}
\frac{1}{r} & \textrm{if } \{i,j\}\in E\\
0 & \textrm{else,} 
\end{array}
\right.
\end{eqnarray}
where $r:=2|E|/|V|$ denotes the vertex degree.  Then, our mixing condition (iii)  holds as soon as $G$ has \emph{moderate growth},  in a  sense that was introduced by Diaconis and Saloff-Coste in \cite{MR1254308} and  later substantially simplified by Tessera and Tointon \cite{MR4253426}. More precisely, if
\begin{eqnarray}
\label{def:moderate}
|E|\ \le\ \theta|V| & \textrm{ and } & \diam(G)\ \ge\ |V|^\delta,
\end{eqnarray}
for some  $\theta,\delta\in(0,\infty)$, then  \cite[Proposition 2.3]{berestycki2026universalityfluctuationscovertime} (which builds upon   \cite{MR1254308,MR4253426}) ensures that 
\begin{eqnarray}
M_u(\bP) & \le & C_{u,\delta,\theta},
\end{eqnarray}
where $C_{u,\delta,\theta}<\infty$ depends only on $u,\delta,\theta$. Consequently, we obtain the following result:
\begin{corollary}[Cutoff on graphs of moderate growth]\label{co:moderate}Let $\bP$ be as in (\ref{def:SRW}), where $G$  depends on a parameter $n\in\dN$ in such a way that $|V|\to\infty$ as $n\to\infty$, and that  (\ref{def:moderate}) holds for some  $\delta,\theta\in(0,\infty)$ that do not depend on $n\in\dN$. Then, for any $k\in\left[|V|^{1-o(1)},|V|-|V|^{1-o(1)}\right]$,  
\begin{eqnarray}
\forall\varepsilon\in(0,1),\quad \frac{\tmix(\varepsilon)}{\trel(\bP)\log|V|}  & \xrightarrow[n\to\infty]{} & \frac{1}2.
\end{eqnarray}
\end{corollary}
More generally, our rate condition (iii) is clearly implied by an estimate of the form
\begin{eqnarray}
\label{eq:heat-kernel}
\max_{i\in[N]}\bP_t(i,i) & \le & \frac{A}{N}
\left(\frac{\trel(\bP)}{t}\right)^\beta,
\qquad 0<t\le\trel(\bP),
\end{eqnarray}
with $A,\beta>0$ independent of $n$, and such a heat-kernel bound  classically follows from an appropriate Nash inequality ~\cite{Nash}. Further examples include long-range walks on finite
tori and suitable finite quotients of groups of polynomial growth, subject
to the corresponding finite-volume estimates \cite{CKSCWZLongRange,SaloffCosteZhengSpreadOut}, as well as finite fractal approximations under the analytic
hypotheses of~\cite{DemboKumagaiNakamura}.
On the other hand,  our methods do not cover high-dimensional geometries such as expanders or the Boolean hypercube, and we leave their analysis to future work. 

\subsection{Proof outline}
As often when estimating mixing times, the proof of (\ref{conclusion}) is split into two unequal halves: a fairly easy lower bound, and a much more demanding upper bound. The lower bound
\begin{eqnarray}
\label{LB}
\liminf_{n\to\infty}\left\{\frac{2\tmix(\varepsilon)}{\trel(\bP) \log N}\right\}  & \ge  & 1,
\end{eqnarray}
is established in Section \ref{sec:LB}  by implementing a classical   idea due to Wilson, which consists in using an eigenfunction of the generator as a distinguishing statistic \cite{MR2023023}. As observed in \cite{MR3069380}, this strategy is particularly adapted to the case of the Exclusion Process,  because eigenfunctions can be produced in a simple and explicit way by lifting those of the transition matrix $\bP$. Moreover, the usually delicate second-moment computation is here greatly simplified  by the observation that the Exclusion Process has negative correlations \cite{NegCorr}. As a consequence, the analysis  boils down to establishing a form of spatial delocalization for the dominant eigenvector of $\bP$, which turns out to follow rather easily from our mixing requirement. 

For the upper bound, we use the powerful comparison theory for Markov chains introduced by Diaconis and Saloff-Coste  \cite{MR1245303,comparison}. Let us   recall  that the \emph{Dirichlet form} associated with a symmetric generator $\cL$ on a finite set $\Omega$ is the quadratic form on $L^2(\Omega)$ defined by
\begin{eqnarray}
\cE(f,f) & := &\langle f,-\cL f\rangle \ = \ \frac{1}{2}\sum_{\omega,\omega'\in\Omega}\cL(\omega,\omega')\left[f(\omega)-f(\omega')\right]^2. 
\end{eqnarray}
The idea is to relate the Dirichlet form $\cE_{\bP,k}^\ex$ associated with the Exclusion generator (\ref{def:L}) to that of its much better-understood \emph{mean-field version} $\cE_{\bK,k}^\ex$, obtained by replacing $\bP$ with
\begin{eqnarray}
\label{def:K}\bK(i,j) & := & \frac 1N.
\end{eqnarray}
In this context, the best inequality one could hope for is
\begin{eqnarray}
\label{conjecture}
 \cE_{\bK,k}^\ex  & \le & \trel(\bP)\,\cE_{\bP,k}^\ex. 
\end{eqnarray}
Such a sharp comparison of quadratic forms is known to hold in another fundamental model of interacting particles: the \emph{Zero-Range Process} \cite{HS}. In the case of the Exclusion Process, an approximate version of (\ref{conjecture}) was established in \cite{MR4164852}, with a multiplicative error $C$ which, in our regime of interest, is independent of $n$. This is already enough to imply
\begin{eqnarray}
\label{UB}
\limsup_{n\to\infty}\left\{\frac{2\tmix(\varepsilon)}{\trel(\bP) \log N}\right\} & \le  & {C},
\end{eqnarray}
but insufficient to match (\ref{LB}) and deduce cutoff.

 Building upon this observation, we introduce a natural  variant of the Exclusion Process, which we coin the \emph{Jump-Over Exclusion Process}.  In the latter, the particles still attempt to jump independently according to  $\bP$, but whenever a particle jumps to an occupied site, it instantaneously jumps again from there, and so on, until it lands onto an empty site. This accelerated dynamics turns out to satisfy the comparison principle (\ref{conjecture}) exactly, as shown in Section \ref{sec:Jump-Over}. By combining this with the celebrated Octopus Inequality of Caputo, Liggett and Richthammer \cite{MR2629990}, we deduce in Section \ref{sec:Octopus} that  our original Exclusion Process also satisfies (\ref{conjecture}), up to a multiplicative error which now crucially reduces to $1+o(1)$, but only in  the low-density regime where $k=o(N)$. In Section \ref{sec:RT}, we bypass this  limitation by showing that  the spectrum of $\cL^\ex_{\bP,k}$ is in fact the superposition, in an appropriate sense, of the spectra of $\cL^\ex_{\bP,\ell}$ for all $\ell\in\{1,\ldots,k\}$, and that the dominant contribution for mixing actually comes from the regime where $\ell=o(N)$. This allows us to prove (\ref{UB}) with $C=1$. Finally, in Section \ref{sec:window}, we refine the proof to obtain an estimate on the cutoff window.
 
\paragraph{Update.}A few days before the submission of the present paper, a very different proof of Corollary \ref{co:tori} was posted on arXiv \cite{chen2026canonicallocalequilibriumcutoff}. In comparison with ours, the approach therein has the advantage of being precise enough to yield a Gaussian cutoff profile on $\dZ_n^d$. On the other hand, our approach is significantly shorter, applies to a broader class of geometries, and ensures that the chain is mixed in the stronger $\chi^2-$divergence sense.

\paragraph{AI-use statement.} AI was not used here: the ideas, the proofs and the writing are ours.

\paragraph{Acknowledgment.}The authors acknowledge support from the NSERC (J.H.), the ISF and the BSF (G.K.) and the ERC (J.S.). 

\section{Proofs}
Let us start with two classical observations that will be used several times in the proof. The first is that the image of the $k-$particle Exclusion Process under the map $A\mapsto A^c$ is a $(N-k)-$particle Exclusion Process with the same rates.  In view of this well-known \emph{particle-hole symmetry}, we may and will henceforth assume that 
\begin{eqnarray}
\label{particle-hole}
k & \le &  N/2.
\end{eqnarray}
The second observation is that for any $t\ge 0$, we have
\begin{eqnarray}
\label{diagonal}
\max_{i\in[N]}N\bP_{t}(i,i)-1 \ = \ \max_{i,j\in[N]}\left|N\bP_{t}(i,j)-1\right| & = & \max_{i\in [N]}N\sum_{j=1}^N\left[\bP_{\frac{t}{2}}(i,j)-\frac{1}{N}\right]^2,
\end{eqnarray}
as can be seen by writing $\bP_{t}=\bP_{t/2}\bP_{t/2}$ and applying the Cauchy-Schwarz inequality. 
\subsection{Wilson's method}
\label{sec:LB}
\begin{proposition}[Lower bound on the mixing time]\label{pr:LB}For any $\varepsilon\in(0,1)$ and $u>0$,
\begin{eqnarray}
\tmix(\varepsilon) & \ge & \frac{1}{2}\trel(\bP) \log\left[\frac{(1-\varepsilon)e^2  ke^{2u}}{32\varepsilon M_u^2(\bP)}\right].\end{eqnarray}
In particular, in the regime of Theorem \ref{th:main}, 
\begin{eqnarray}
\liminf_{n\to\infty}\left\{\frac{2\tmix(\varepsilon)}{\trel(\bP) \log N}\right\}  & \ge  & 1.
\end{eqnarray}
\end{proposition}
\begin{proof}
In order to prove that two probability measures $\mu,\nu$ on a finite space $\Omega$ are far apart in total variation, it is classically enough to find a  test function $f\in\dR^\Omega$ -- often referred to as a \emph{distinguishing statistic} -- whose expectations are very different under $\mu$ and $\nu$, compared to the corresponding standard deviations. More precisely, we  have
\begin{eqnarray}
\label{distinguishing}
\dtv(\mu,\nu) & \ge & \left(1+2\frac{\Var_\mu [f]+\Var_\nu[f]}{\left|\EE_\mu[f]-\EE_\nu[f]\right|^2}\right)^{-1},
\end{eqnarray}
see, e.g., \cite{MR3726904}. Following a celebrated idea due to Wilson \cite{MR2023023}, we here use a leading eigenfunction of the generator as our distinguishing statistic.  Specifically, we let $\phi\in\dR^N$ be an eigenvector of $\bP$ corresponding to the second largest eigenvalue $\lambda_2$, normalized so that $\|\phi\|_\infty=1$, and we define a test function $f\in\dR^{\Omega_k}$ as follows:
\begin{eqnarray}
\forall A\in\Omega_k,\qquad f(A) & := & \sum_{i\in A}\phi(i).
\end{eqnarray}
It then easily follows from the definition of $\cL^\ex_{\bP,k}$ that
$
\cL^\ex_{\bP,k} f =  -\gamma f, 
$
where $\gamma=1-\lambda_2$ is the spectral gap of $\bP$. Consequently, letting $(\cS_t)_{t\ge 0}$ denote a $k-$particle Exclusion Process with rates $\bP$ starting from a non-random initial state $\cS_0\in\Omega_k$, we have
\begin{eqnarray}
\label{expectation}
\EE\left[f(\cS_t)\right] & = & e^{-\gamma t}f(\cS_0),
\end{eqnarray}
for all $t\ge 0$. On the other hand,  the events $\left(\{i\in \cS_t\}\right)_{1\le i \le N}$ are  well known to be pairwise negatively correlated (see \cite{NegCorr} for a proof,  and \cite{Rayleigh} for a much stronger property), so that
\begin{eqnarray*}
\label{variance}
\Var\left[f(\cS_t)\right] & \le & 2\Var\left[\sum_{i\in\cS_t}\phi_+(i)\right]+2\Var\left[\sum_{i\in\cS_t}\phi_-(i)\right]\\
& \le & 2\sum_{i=1}^N\phi^2_+(i)\PP(i\in \cS_t)\PP(i\notin \cS_t)+2\sum_{i=1}^N\phi^2_-(i)\PP(i\in \cS_t)\PP(i\notin \cS_t) \\
& \le & 2\|\phi\|_\infty^2 \sum_{i=1}^N\PP(i\in \cS_t) \ = \ 2k.
\end{eqnarray*}
In view of those estimates and their $t\to\infty$ counterparts, the bound (\ref{distinguishing}) yields
\begin{eqnarray}
\dd(t) & \ge & \left(1+\frac{8ke^{2\gamma t}}{f^2(\cS_0)}\right)^{-1}.\end{eqnarray}
Choosing $t=\tmix(\varepsilon)$, and recalling that $\trel(\bP)=1/\gamma$, we arrive at 
\begin{eqnarray}
\label{LB1}
\tmix(\varepsilon) & \ge & \frac{1}{2}\trel(\bP) \log\left[\frac{(1-\varepsilon)f^2(\cS_0)}{8\varepsilon k}\right].
\end{eqnarray}
It now remains to optimize on the choice of the initial set $\cS_0\in\Omega_k$. Because of (\ref{particle-hole}) and upon replacing $\phi$ by $-\phi$ if necessary, we may assume that $\phi$ has at least $k$ non-negative entries, and the set $\cS_0$ corresponding to the $k$ largest entries then clearly satisfies 
\begin{eqnarray}
\label{LB2}
f(\cS_0) & \ge & \frac{k}{2N}\sum_{j=1}^N|\phi(j)|. 
\end{eqnarray}
On the other hand,  for any $i\in[N]$ and $t\ge 0$, we can use $\bP_t \phi=e^{-\gamma t}\phi$ and (\ref{diagonal}) to write
\begin{eqnarray*}
e^{-\gamma t}|\phi(i)| & = & \left|\sum_{j=1}^N\bP_t(i,j)\phi(j)\right|
\ \le \  \left(\max_{i\in[N]}\bP_t(i,i)\right)\sum_{j=1}^N|\phi(j)|.
\end{eqnarray*}
Setting $t:=u\trel(\bP)$ and choosing $i\in[N]$ such that $|\phi(i)|=\|\phi\|_\infty=1$, we deduce that
 \begin{eqnarray}
\frac{e^{-u}}{M_u(\bP)}  & \le & \frac{1}{N}\sum_{j=1}^N|\phi(j)|.
\end{eqnarray}
Re-inserting this back into (\ref{LB2})  and then into (\ref{LB1}) concludes the proof. 
\end{proof}

\subsection{The Jump-Over Exclusion Process}

\label{sec:Jump-Over}
Let  $X=(X_0,X_1,\ldots)$ be a discrete-time Markov chain with transition matrix $\bP$, and write  $\tau_0<\tau_1<\cdots$ for the successive times at which $X$ lies in a set $A\subseteq[N]$. Then, the \emph{induced process} $(X_{\tau_0},X_{\tau_1},\ldots)$ is clearly a Markov chain on $A$, with (symmetric) transition matrix
\begin{eqnarray*}
\forall i,j\in A,\qquad \bP^A(i,j) & := & \PP_i(X_{\tau_1}=j).
\end{eqnarray*}
For any integer $0<k<N$, we define the $k-$particle \emph{Jump-Over Exclusion Process} with rates $\bP$ as the continuous-time Markov chain with state space $\Omega_k$ and  generator 
\begin{eqnarray}
\cL^{\textsc{jo}}_{\bP,k} f(A)& := & \sum_{i\in A}\sum_{j\in A^c}\bP^{(A^i)^c}(i,j)\left[f(A^{ij})-f(A)\right],
\end{eqnarray}
where   $A^i$  denotes the symmetric difference of $A$ and $\{i\}$.
Just like $\cL^{\textsc{ex}}_{\bP,k}$, this  generator  describes the  evolution of $k$ non-overlapping particles that attempt to jump independently according to   $\bP$. The difference is that  instead of being  canceled, any jump to an occupied site is here instantaneously followed by another jump from there, and so on, until an empty site is found. The interest of this accelerated variant  lies in the fact that  the associated Dirichlet form $\cE^{\textsc{jo}}_{\bP,k}$
satisfies a perfect comparison  with its mean-field version. This is the content of the following result,  inspired by an analogous one for the Zero-Range Process \cite{HS}. 
\begin{proposition}[Perfect mean-field comparison for the Jump-Over Exclusion Process]\label{pr:MFJO}
\begin{eqnarray}
\label{dir}
\cE_{\bK,k}^{\textsc{jo}} & \le & \trel(\bP)\,\cE_{\bP,k}^\textsc{jo}.
\end{eqnarray}
\end{proposition}
\begin{proof}
An elementary but crucial property of the induced chain is that it satisfies $\trel(\bP^A)\le \trel(\bP)$, for any  non-empty $A\subseteq[N]$ (see, e.g., \cite[Theorem 13.16]{MR3726904}). In other words, 
\begin{eqnarray}
\label{trace}
\frac{1}{|A|}\sum_{i,j\in A}\left[g(i)-g(j)\right]^2 & \le & \trel(\bP)\sum_{i,j\in A}\bP^A(i,j) \left[g(i)-g(j)\right]^2,
\end{eqnarray}
for any $g\in\dR^A$. In particular, for any $f\in L^2(\Omega_k)$, we can write
\begin{eqnarray*}
\cE_{\bP,k}^\textsc{jo}(f,f) & = & \frac{1}{2}\sum_{A\in\Omega_k}\sum_{i\in A}\sum_{j\in A^c}\bP^{(A^i)^c}(i,j)\left[f(A^{ij})-f(A)\right]^2\\
& = & \frac{1}{2}\sum_{B\in\Omega_{k-1}}\sum_{i,j\in B^c}\bP^{B^c}(i,j)\left[f(B^j)-f(B^i)\right]^2\\
& \ge & \frac{1}{2\trel(\bP)}\sum_{B\in\Omega_{k-1}}\frac{1}{|B^c|}\sum_{i,j\in B^c}\left[f(B^j)-f(B^i)\right]^2\\
& = & \frac{1}{\trel(\bP)}\cE_{\bK,k}^\textsc{jo}(f,f),
\end{eqnarray*}
 where the second line just uses the change of variable $B:=A^i$, while the third uses (\ref{trace})  with $A:=B^c$ and $g(i):= f(B^i)$.
\end{proof}
By a naive comparison argument, we  will deduce an analogue of Proposition \ref{pr:MFJO} for the original Exclusion Process. There is of course a  price to pay, and the latter is small only when 
$\bP$ is sufficiently close to $\bK$ and 
 the particle density $k/N$ is low, 
a regime which seems rather far from the one we  care about. Those  limitations will be bypassed in the next sections.  
\begin{corollary}[Approximate mean-field comparison for the Exclusion Process]\label{co:MFEX}
\begin{eqnarray}
\trel(\bP)\,\cE^{\textsc{ex}}_{\bP,k} & \ge & \left(1-\delta(\bP,k)\right)\,\cE^{\textsc{ex}}_{\bK,k},
\end{eqnarray}
where the error term $\delta(\bP,k)$ is defined as follows:
\begin{eqnarray}
\label{def:delta}
\delta(\bP,k) & := & \frac{2k}{N}\left(N\max_{1\le i,j\le N}\bP(i,j)\right)^2\trel(\bP).
\end{eqnarray}
\end{corollary}
\begin{proof} We may assume that $\delta(\bP,k)\le 1$, otherwise there is nothing to prove. For any set $A\in\Omega_k$ and any $(i,j)\in A\times A^c$, we have by definition
\begin{eqnarray*}
\bP^{(A^i)^c}(i,j) - \bP(i,j)  & = &  \sum_{m=1}^\infty\sum_{i_1,\ldots,i_m\in A\setminus i}\bP(i,i_1)\bP(i_1,i_2)\cdots \bP(i_{m-1},i_m)\bP(i_m,j),
\end{eqnarray*}
In particular, crudely bounding each entry of $\bP$ in the sum by $p:=\max_{i,j\in[N]}\bP(i,j)$ yields
\begin{eqnarray}
0 \ \le &  \bP^{(A^i)^c}(i,j)-\bP(i,j) & \le \ 2kp^2,
\end{eqnarray}
where we have used $2kp\le \delta(\bP,k)\le 1$. Now, for any $f\in L^2(\Omega_k)$, we may multiply both sides by $\frac 12[f(A^{ij})-f(A)]^2$ and sum over all $(i,j)\in A\times A^c$ and  $A\in\Omega_k$ to conclude that
\begin{eqnarray}
0 \ \le  & \cE^{\textsc{jo}}_{\bP,k}-\cE^{\textsc{ex}}_{\bP,k} & \le \   \frac{\delta(\bP,k)}{\trel(\bP)}\cE^\ex_{\bK,k},
\end{eqnarray}
where we have used the fact that $2Nkp^2=\delta(\bP,k)/\trel(\bP)$. 
Of course, the first inequality also applies to $\bP=\bK$, and the desired result now readily follows from Proposition \ref{pr:MFJO}. 
\end{proof}

\subsection{The Octopus Inequality}
\label{sec:Octopus}
The celebrated \emph{Octopus Inequality} of Caputo, Liggett and Richthammer \cite{MR2629990} concerns a natural labeled refinement of the Exclusion Process known as the \emph{Interchange Process}, with state space the symmetric group $\fS_N$ and generator 
\begin{eqnarray}
\cL^{\textsc{ip}}_\bP f(\sigma) &:= & \frac{1}{2}\sum_{1\le i,j\le N}\bP(i,j)\left[f(\tau_{ij}\sigma)-f(\sigma)\right],
\end{eqnarray}
where $\tau_{ij}\in\fS_N$ is the transposition of $i$ and $j$ (or the identity if $i=j$). Using the Octopus inequality, we can compare the associated Dirichlet form $\cE^\ip_\bP$ to $\cE^\ip_{\bP_t}$ for any time $t\ge 0$, where we recall that $(\bP_t)_{t\ge 0}$ is the continuous-time semi-group induced by $\bP$, as defined at (\ref{def:Pt}). We believe that this result is new, and interesting in its own right. It is inspired by a computation made in \cite{MR4164852} and \cite{MR4254474}, which we improve and build upon. 
\begin{proposition}[Dynamical comparison principle]\label{pr:octopus}For all $t\ge 0$, we have
\begin{eqnarray}
 \cE_{\bP_t}^\ip &\le & t\cE_\bP^\ip.
\end{eqnarray}
In particular, the same is true for the $k-$particle Exclusion Process. 
\end{proposition}
\begin{proof}
The Dirichlet form associated with $\cL^{\textsc{ip}}_\bP$ is $\cE^\ip_\bP =  \frac{1}{2}\sum_{ i,j\in[N]}\bP(i,j)\,\cE_{i,j}^\ip$, where
\begin{eqnarray*}
\cE_{i,j}^\ip(f,f) & := & \frac{1}{2}\sum_{\sigma\in\fS_N}[f(\tau_{ij}\sigma)-f(\sigma)]^2,
\end{eqnarray*}
represents the contribution from the edge $\{i,j\}$. Note that this quantity is symmetric in $i$ and $j$ and is zero when $i=j$. With this notation at hand, the Octopus Inequality \cite{MR2629990} reads
\begin{eqnarray}
\left(1-\bP(o,o)\right)\sum_{i=1}^N\bP(o,i)\cE_{o,i}^\ip & \ge & \frac{1}{2}\sum_{i,j\in[N]\setminus \{o\}}\bP(o,i)\bP(o,j)\cE_{i,j}^\ip,
\end{eqnarray}
for any $o\in[N]$. Moving the negative term to the right-hand side and re-arranging, we obtain
\begin{eqnarray}
\sum_{i=1}^N\bP(o,i)\cE_{o,i}^\ip & \ge & \frac{1}{2}\sum_{i,j\in[N]}\bP(o,i)\bP(o,j)\cE_{i,j}^\ip,
\end{eqnarray}
and summing this  over  $o\in[N]$ yields a discrete-time version of our target inequality:
\begin{eqnarray}
\label{octopus:discrete}
2\cE_\bP^\ip & \ge & \cE^\ip_{\bP^2}.
\end{eqnarray}
 The same inequality was obtained in \cite{MR4164852,MR4254474}, but under additional assumptions on $\bP$. Now, fix a time $t\ge 0$ and an integer $n\ge \log_2 t$, and consider the symmetric stochastic matrix 
\begin{eqnarray}
\mathrm{Q} & := & \left(1-\frac{t}{2^{n}}\right)\mathrm{Id}+\frac{t}{2^{n}}\bP.
\end{eqnarray}
Applying (\ref{octopus:discrete}) to $\bQ$ and iterating $n$ times, we arrive at
\begin{eqnarray}
2^n\cE_\bQ^\ip & \ge & \cE^\ip_{\bQ^{2^n}}.
\end{eqnarray}
But the left-hand side is exactly $t\cE^\ip_{\bP}$, while the right-hand side tends to $\cE^\ip_{\bP_t}$ as $n\to\infty$, establishing the first claim. Finally, recall that the $k-$particle Exclusion Process is the image of the Interchange Process under the projection map $\Phi_k\colon \fS_N\to\Omega_k$ defined by $\Phi_k(\sigma):= \sigma(\{1,2,\ldots,k\})$. More precisely, for any $f\in L^2(\Omega_k)$, we have
\begin{eqnarray}
(\cL^\ex_{\bP,k} f)\circ\Phi_k & = &  \cL^\ip_\bP (f\circ \Phi_k).
\end{eqnarray}
Since each $A\in\Omega_k$ has exactly $k!(N-k)!$ pre-images under $\Phi_k$, this implies
\begin{eqnarray}
k!(N-k)!\,\cE^\ex_{\bP,k}(f,f)   & = & \cE^\ip_\bP(f\circ \Phi_k,f\circ\Phi_k).
\end{eqnarray}
Thus, the conclusion automatically extends to the Exclusion Process.
\end{proof}
As an application, we obtain the following refined mean-field comparison principle, where the multiplicative error now remains bounded in the regime we consider, and crucially tends to $1$ in the low-density regime $k=o(N)$.
\begin{corollary}[Refined comparison]\label{co:refined}Let $(N,k,\bP)$ depend on  $n\in\dN$ in such a way that
\begin{eqnarray}
\forall\varepsilon>0,\quad \limsup_{n\to\infty} M_\varepsilon(\bP)<\infty.
\end{eqnarray}
Then, there is a constant $C<\infty$, independent of $n$, such that 
\begin{eqnarray}
 \cE^\ex_{\bK,k} & \le & C\,\trel(\bP)\,\cE^\ex_{\bP,k}.
\end{eqnarray}
 If moreover $k=o(N)$ as $n\to\infty$, then we have the much sharper estimate
\begin{eqnarray}
\cE^\ex_{\bK,k} & \le & (1+o(1))\,\trel(\bP)\,\cE^\ex_{\bP,k}.
\end{eqnarray}

\end{corollary}
\begin{proof}
The quantity on the right-hand side of (\ref{diagonal}) is the $\chi^2-$divergence w.r.t. equilibrium of the continuous-time dynamics (\ref{def:Pt}), at time $t/2$.  This function of $t$ is well known to decay exponentially at rate $\gamma=1/\trel(\bP)$. Moreover, its value at $t=\trel(\bP)$ is $M_1(\bP)-1$, which is assumed to be bounded  independently of $n$. Thus, its value at $t=\theta \trel(\bP)$ can be made less than $1/2$ by choosing $\theta$ large enough, independently of $n$. For that choice of $t$, we have
\begin{eqnarray}
\forall i,j\in[N], \qquad \bP_{t}(i,j) & \ge & \frac{1}{2N},
\end{eqnarray}
which trivially implies that $\cE_{\bK,k}\le 2\cE_{\bP_{t},k}^\ex$. Applying Proposition \ref{pr:octopus},  we obtain the first claim with $C=2\theta$. To prove the second claim, we first apply Corollary \ref{co:MFEX}  to $\bP_t$ instead of $\bP$ and then use Proposition \ref{pr:octopus} to deduce that for any $t\ge 0$,
\begin{eqnarray}
\label{MFt}
\left[{1-\delta(\bP_t,k)}\right] \cE_{\bK,k}^\ex & \le & {\trel(\bP_t)t}\,\cE_{\bP,k}^\ex.
\end{eqnarray}
 Let us now fix $u>0$ and choose $t=u\trel(\bP)$. We then have
\begin{eqnarray*}
\trel(\bP_t) \ = \ \frac{1}{1-e^{-u}},\qquad
N \max_{1\le i,j\le N}\bP_t(i,j) \ = \  M_u(\bP),\qquad
\delta(\bP_t,k) \ = \ \frac{2kM_u^2(\bP)}{N(1-e^{-u})}.
\end{eqnarray*}
Our assumptions readily imply that  $\delta(\bP_t,k)=o(1)$ as $n\to\infty$, so that (\ref{MFt}) yields
\begin{eqnarray*}
\left[\frac{1-e^{-u}}{u}-o(1)\right]\cE_{\bK,k}^\ex & \le & {\trel(\bP)}\,\cE_{\bP,k}^\ex.
\end{eqnarray*}
Finally, note that the constant on the left-hand side can be made arbitrarily close to $1$ by choosing $u$ small enough.
\end{proof}

\subsection{Fourier analysis}
\label{sec:RT}

Mixing times of reversible Markov chains are classically controlled by the spectrum of their generator, and the refined comparison principle established above allows one to relate the spectrum of $\cL^\ex_{\bP,k}$ to the (fully explicit) spectrum of its mean-field version $\cL^\ex_{\bK,k}$. Unfortunately, this comparison is precise enough only in the low-density regime where $k=o(N)$. To bypass this limitation, we will now argue that  the spectrum of $\cL^\ex_{\bP,k}$ is in fact the superposition, in an appropriate sense, of the spectra of $\cL^\ex_{\bP,\ell}$ for all $\ell\in\{1,\ldots,k\}$, and that the dominant contribution for mixing actually comes from the regime where $\ell=o(N)$. 
\begin{theorem}[Canonical decomposition]\label{th:fourier}The space $L^2(\Omega_k) $ admits a canonical (meaning, depending only on $N$ and $k$) orthogonal decomposition of the form
\begin{eqnarray}
\label{canonical}
L^2(\Omega_k) & = & \bigoplus_{\ell=0}^k\cH_{k, \ell}
\end{eqnarray}
where $\cH_{k,0}$ is the space of constant functions and for any  $\ell\in\{1,\ldots,k\}$, 
\begin{enumerate}
 \item the dimension of  $\cH_{k,\ell}$ is  
$
d_\ell :=  {N\choose \ell}-{N\choose \ell-1};
$
  \item $\cH_{k,\ell}$ is invariant under the natural action of $\fS_N$, hence in particular under  $\cL^{\ex}_{\bP,k}$;
\item the spectrum of the restriction of $\cL^{\ex}_{\bP,k}$ to $\cH_{k,\ell}$ is independent of $k$. 
\end{enumerate}
\end{theorem}
\begin{proof}Any $f\in L^2(\Omega_{k-1})$ can be  \emph{lifted} to produce a function $U_k f\in L^2(\Omega_{k})$ via
\begin{eqnarray}
U_k f(A) & := & \sum_{i\in A}f(A\setminus\{i\}).
\end{eqnarray}
Moreover, this transformation is injective as long as $k\le N/2$, since one can check that
\begin{eqnarray*}
f(A) & = & \sum_{B\in\Omega_{k}}c_{|A\cap B|}\, U_k f(B), \quad\textrm{ where }\quad c_r\ :=\ \frac{(-1)^{k-1-r}}{k-r}{N+1-k\choose k-r}^{-1}.
\end{eqnarray*}
Dually, we may \emph{downgrade} any function $g\in L^2(\Omega_{k})$ into one in $L^2(\Omega_{k-1})$ via
\begin{eqnarray}
U_k^\star g(A) & := & \sum_{i\in A^c}g(A\cup\{i\}).
\end{eqnarray}
The notation is justified by the fact that $U_k^\star\colon L^2(\Omega_{k})\to L^2(\Omega_{k-1}) $ is the adjoint of $U_k$, i.e.
\begin{eqnarray}
\forall (f,g)\in L^2(\Omega_{k-1})\times L^2(\Omega_{k}),\qquad \langle U_k f,g\rangle_{L^2(\Omega_{k})} & = & \langle  f,U_k^\star g\rangle_{L^2(\Omega_{k-1})}.
\end{eqnarray}
as is readily checked. This yields the key  orthogonal decomposition
\begin{eqnarray}
 L^2(\Omega_k) & = &  \im(U_k)\oplus \ker(U_k^\star).
 \end{eqnarray} 
Thus, the desired decomposition of $L^2(\Omega_k)$ can be simply defined recursively by lifting the decomposition   of $L^2(\Omega_{k-1})$ through $U_k$, and then adding  its orthogonal complement:
\begin{eqnarray}
\label{rule}
\left[\cH_{k,0} ,\ldots,\cH_{k,k} \right] & := & \left[U_k(\cH_{k-1,0}),\ldots,U_k(\cH_{k-1,k-1}),\ker(U_k^\star)\right].
\end{eqnarray}
The orthogonality of $\left[\cH_{k-1,0},\ldots,\cH_{k-1,k-1}\right]$ is preserved under $U_k$, thanks to the identity
\begin{eqnarray}
\forall f\in\cH_{k-1,\ell},\qquad  U_k^\star U_{k} f & = & (k-\ell)(N-k-\ell+1)f,
 \end{eqnarray} 
which can be checked by an easy induction over $k>\ell$. Finally, let us verify the three claims in the theorem by induction. First, for any $\ell\in\{0,\ldots,k-1\}$ we have
\begin{eqnarray}
\dim(\cH_{k,\ell}) & = & \dim(\cH_{k-1,\ell}) \ = \ d_\ell,
\end{eqnarray}
by the  injectivity of $U_k$  and the induction hypothesis, and it follows that  
\begin{eqnarray}
\dim(\cH_{k,k}) & = & {N \choose k}- \left(d_0+\cdots+d_{k-1}\right) \ = \ d_k.
\end{eqnarray}
Second, if $T^\sigma_k$ denotes the natural action of a permutation $\sigma\in\fS_N$ on $L^2(\Omega_k)$ given by
\begin{eqnarray}
(T^\sigma_k f)(A) & := & f(\sigma(A)),
\end{eqnarray}
then the stability of $\cH_{k,\ell}$ under  $T^\sigma_k$ follows inductively from the  intertwining relation
\begin{eqnarray}
 T^\sigma_k  U_{k} & = & U_{k}  T^\sigma_{k-1}.
\end{eqnarray}
Finally, in view of the expression 
\begin{eqnarray}
\cL_{\bP,k}^\ex & = & \frac{1}{2}\sum_{i,j\in[N]}\bP(i,j)(T^{(\tau_{ij})}_k -\mathrm{Id}_k),
\end{eqnarray}
the operator $\cL^\ex_{\bP,k}$ inherits the same properties: it leaves $\cH_{k,\ell}$ invariant and 
\begin{eqnarray}
 \cL_{\bP,k}^\ex  U_{k} & = & U_{k}  \cL_{\bP,k-1}^\ex.
\end{eqnarray}
Thus, any eigenfunction of $\cL_{\bP,k-1}^\ex$ in $\cH_{k-1,\ell}$  can be lifted through $U_{k}$ to yield an eigenfunction of  $\cL_{\bP,k}^\ex$ in $\cH_{k,\ell}$ with the same eigenvalue, establishing the third claim.
\end{proof}
With  Theorem \ref{th:fourier} at hand, we can now define, for each $\ell\le N/2$,  the crucial quantity
\begin{eqnarray}
\label{def:gl}
\gamma_\ell(\bP) & := & \inf_{f\in\cH_{\ell,\ell}\setminus\{0\}}\left\{\frac{\cE^\ex_{\bP,\ell}(f,f)}{\|f\|^2}\right\},
\end{eqnarray}
where $\|f\|^2=\sum_{A\in\Omega_\ell}f^2(A)$ denotes the usual norm on $L^2(\Omega_\ell)$.
 We will use $\gamma_1(\bP),\ldots,\gamma_k(\bP)$ to control the \emph{$\chi^2-$divergence} to equilibrium of our $k-$particle Exclusion Process. For any probability measures $\mu,\pi$ on a finite state space $\Omega$, we define
\begin{eqnarray}
\chi^2(\mu\,|\,\pi) & := & \sum_{\omega\in\Omega}\pi(\omega)\left|\frac{\mu(\omega)}{\pi(\omega)}-1\right|^2,
\end{eqnarray}
and we note that $\chi^2(\mu\,|\,\pi)\ge 4\dtv^2(\mu,\pi)$, by the Cauchy-Schwarz inequality.
\begin{corollary}[Spectral bound on the $\chi^2-$divergence to equilibrium]\label{co:Fourier}Let $(\cS_t)_{t\ge 0}$ denote a $k-$particle Exclusion Process with rates $\bP$, started from any initial state $\cS_0\in\Omega_k$. Then, 
\begin{eqnarray*}
\forall t\ge 0,\qquad \chi^2(\mathrm{Law}(\cS_t)\,|\,\pi) & \le & \sum_{\ell=1}^k d_\ell e^{-2\gamma_\ell(\bP)t}.
\end{eqnarray*}
\end{corollary}
\begin{proof}
Letting $f_t\colon A\mapsto \PP(\cS_t=A)$ be the probability mass function of $\cS_t$, we have
\begin{eqnarray}
\chi^2(\mathrm{Law}(\cS_t)\,|\,\pi) & = & |\Omega_k|\left\|f_t\right\|^2-1,
\end{eqnarray}
For each $\ell\in\{0,\ldots,k\}$, consider an orthonormal basis $(\phi_{\ell,1},\ldots,\phi_{\ell,d_\ell})$ of $\cH_{k,\ell}$ consisting of eigenvectors of $-\cL^{\ex}_{\bP,k}$ with eigenvalues $(\gamma_{\ell,1},\ldots,\gamma_{\ell,d_\ell})$.  Since  $f_t=e^{t\cL^\ex_{\bP,k}} f_0$, we have  
\begin{eqnarray}
\left\|f_t\right\|^2 & = &  \sum_{\ell=0}^k\sum_{i=1}^{d_\ell}e^{-2t\gamma_{\ell,i}}\langle f_0,\phi_{\ell,i}\rangle^2.
\end{eqnarray}
Now, the spectrum $\{\gamma_{\ell,1},\ldots,\gamma_{\ell,d_\ell}\}$ does not depend on $k$ by the above theorem, and its minimum is precisely $\gamma_\ell(\bP)$. Moreover, the term $\ell=0$ in the above sum is simply $1/|\Omega_k|$, because $\cH_{k,0}$ is the space of constant functions.   Thus, the last two displays imply
\begin{eqnarray}
\label{CS2}
\chi^2(\mathrm{Law}(\cS_t)\,|\,\pi) & \le & |\Omega_k|\sum_{\ell=1}^ke^{-2t\gamma_{\ell}(\bP)}\sum_{i=1}^{d_\ell}\langle f_0,\phi_{\ell,i}\rangle^2,
\end{eqnarray}
Finally, recall that $f_0(A)$ is $1$ if $A=\cS_0$ and $0$ else, so that
\begin{eqnarray}
\sum_{i=1}^{d_\ell}\langle f_0,\phi_{\ell,i}\rangle^2 & = & \sum_{i=1}^{d_\ell}\phi_{\ell,i}^2(\cS_0).
\end{eqnarray}
On the other hand, the left-hand side is the squared-norm of the projection of $f_0$ onto the space $\cH_{k,\ell}$, and the latter is invariant under the action of $\fS_N$. Therefore, the norm is independent of the choice of $\cS_0\in\Omega_k$, and summing over all such choices yields
\begin{eqnarray}
|\Omega_k|\sum_{i=1}^{d_\ell}\langle f_0,\phi_{\ell,i}\rangle^2 & = & \sum_{\cS_0\in\Omega_k}\sum_{i=1}^{d_\ell}\phi_{\ell,i}^2(\cS_0) \ = \ \sum_{i=1}^{d_\ell}\|\phi_{\ell,i}\|^2 \ = \ d_\ell.
\end{eqnarray}
Inserting this back into (\ref{CS2}) concludes the proof.
\end{proof}
We now have everything we need to conclude the proof of our main theorem. 
\begin{proposition}[Upper bound on the mixing time]\label{pr:UB}Fix $\delta>0$ and set
\begin{eqnarray}
t & := & \frac{1+\delta}2\trel(\bP)\log N.
\end{eqnarray}
Then, in the regime of Theorem \ref{th:main}, we have 
\begin{eqnarray*}
\max_{\cS_0\in\Omega_k}\chi^2(\mathrm{Law}(\cS_t)\,|\,\pi)  & \xrightarrow[n\to\infty]{} & 0.
\end{eqnarray*}
\end{proposition}
\begin{proof}
The estimate in Corollary \ref{co:Fourier} is increasing in $k$, so we may assume that $k=\lfloor N/2\rfloor$.  
Using the definition of $\gamma_\ell(\bP)$ at (\ref{def:gl}) and the first part of Corollary \ref{co:refined}, we know that
\begin{eqnarray}
 \gamma_\ell(\bK) & \le & C\trel(\bP)\gamma_\ell(\bP),
\end{eqnarray}
for each $\ell\le N/2$, where $C<\infty$ is independent of $n,\ell$. Moreover, we classically have
\begin{eqnarray}
\label{MFspectrum}
\gamma_\ell(\bK) & = & \frac{\ell(N-\ell+1)}{N},
\end{eqnarray}
as computed in the seminal work \cite{MR626813} and  recently rediscovered in \cite{MR3647074}. Thus,
\begin{eqnarray*}
\sum_{\ell= L}^{\lfloor N/2\rfloor }d_\ell e^{-2\gamma_\ell(\bP) t} & \le & \sum_{\ell=L}^{\lfloor N/2\rfloor} d_\ell\exp\left(-\frac{\ell(N-\ell+1)}{CN}\log N\right)\\
& \le & 2^N\exp\left(-\frac {L\log N}{2C}\right),
\end{eqnarray*}
for any $L\le N/2$. Choosing $L=\lfloor N/\sqrt{\log N}\rfloor$ is more than enough to make this $o(1)$.  On the other hand, when $\ell\le L$, we may use the second part of Corollary \ref{co:refined} with $k=\ell$ to obtain 
\begin{eqnarray*}
\trel(\bP)\gamma_\ell(\bP) & \ge & (1-o(1))\,\gamma_\ell(\bK),
\end{eqnarray*}
as $n\to\infty$. Using again the explicit formulae for $ \gamma_\ell(\bK)$ and $d_\ell$, we obtain
 \begin{eqnarray*}
\sum_{\ell= 1}^{L}d_\ell e^{-2\gamma_\ell(\bP) t} & \le & \sum_{\ell=1}^{L}d_\ell\exp\left(-(1+\delta-o(1))\frac{\ell(N-\ell+1)}{N}\log N\right)\\
&\le &  \sum_{\ell=1}^{L}\frac{N^{-(\delta-o(1))\ell}}{\ell!}\\
& \le & \exp\left(N^{-\delta+o(1)}\right)-1,
\end{eqnarray*}
which is also $o(1)$ as $n\to\infty$. This concludes the proof.
\end{proof}

\subsection{Cutoff window}\label{sec:window}
\label{sec:window}
In this final section, we show how a refinement of the  proof of Proposition \ref{pr:UB} yields an estimate of order $\trel(\bP)$ for the width of the cutoff window, under a mild additional requirement on the rate matrix (condition (iv) below). Note that the latter clearly holds under the standard heat-kernel estimate (\ref{eq:heat-kernel}), so that it applies to  all examples   in Corollaries \ref{co:tori} and \ref{co:moderate} (see~\cite{HP}).

\begin{proposition}[Cutoff window]\label{pr:window}Let $(N,k,\bP)$ depend on $n$ in such a way that
\begin{enumerate}[(i)]
\item $N\to\infty$ as $n\to\infty$;
\item $\displaystyle{0<\liminf_{n\to\infty }\frac{k}{N}\le \limsup_{n\to\infty}\frac{k}{N}<1}$;
\item $\displaystyle{\limsup_{n\to\infty}M_u(\bP)<\infty}$ for each $u>0$. 
\item For each $\eta>0$, there is $c>0$ independent of $n$ such that $M_{\frac{c}{\log N}}(\bP)=o(N^\eta)$.
\end{enumerate}
Then,  for any fixed
$\varepsilon\in(0,1)$,
\begin{eqnarray}
\left|\tmix(\varepsilon) -\frac12 \trel(\bP)\log N\right| &= & O_\varepsilon(\trel(\bP)).
\end{eqnarray}

\end{proposition}
\begin{proof}
Set $t:=\frac12\trel(\bP)\log N$ and $L:=\lceil N^{1-\frac{1}{4C}}\rceil$, with $C$ as in  Corollary~\ref{co:refined}. Then, for $\ell\ge L$, we have $\trel(\bP)\gamma_\ell(\bP)\ge\ell/(2C)$ and $d_\ell\le \left(\frac{Ne}{\ell}\right)^\ell\le (N^\frac{1}{4C} e)^\ell$, so 
\begin{eqnarray}
\sum_{\ell=L}^{\lfloor N/2\rfloor}
d_\ell e^{-2t\gamma_\ell(\bP)}
& \le & \sum_{\ell=L}^{\lfloor N/2\rfloor}
e^{\ell(1+\frac{1}{4C}\log N-\frac{1}{2C}\log N)} \ = \ o(1).
\end{eqnarray}
On the other hand, for $\ell\le L$, we have by (\ref{MFt}) and (\ref{def:delta}),
\begin{eqnarray}\label{eq:quantitative-spectrum}
\trel(\bP)\gamma_\ell(\bP) & \ge & \gamma_\ell(\bK)
\left[\frac{1-e^{-u}}{u}-\frac{2\ell M_u^2(\bP)}{Nu}\right],
\end{eqnarray}
for  $u>0$. In view of  assumption (iv) and (\ref{MFspectrum}),  we can choose $u=c/\log N$ with $c$ independent of $n$ to make the right-hand side at least $\left(1-\frac{2c}{\log N}\right)\ell$ for  large enough  $n$, so that
\begin{eqnarray}
\sum_{\ell=1}^{L}
d_\ell e^{-2t\gamma_\ell(\bP)} & \le & \sum_{\ell=1}^{L}
d_\ell N^{-\ell} e^{2c \ell} \ \le \ e^{e^{2c}}.
\end{eqnarray}
Consequently, Corollary~\ref{co:Fourier} yields
$
\limsup_{n\to\infty}\max_{\cS_0\in\Omega_k}
\chi^2(\mathrm{Law}(\cS_{t})\,|\,\pi)
 \le  e^{e^{2c}}$,
 and this is enough to prove the upper-bound, since $t\mapsto \chi(\mathrm{Law}(\cS_{t})\,|\,\pi)$ classically decays exponentially at rate $1/\trel(\bP)$ \cite{MR2629990}. 
The lower bound already follows from
Proposition~\ref{pr:LB}.
\end{proof}

\bibliographystyle{plain}
\bibliography{draft}

\end{document}